\documentclass[12pt, reqno]{amsart}

\author[M.~Caprio]{Michele Caprio and Teddy Seidenfeld}
\address{PRECISE Center, Department of Computer and Information Science,
University of Pennsylvania, 3330 Walnut Street, Philadelphia, PA 19104}
\email{caprio@seas.upenn.edu}
\urladdr{\url{https://precise.seas.upenn.edu/people/doctoral-students}}  
\address{Departments of Philosophy and Statistics, Carnegie Mellon University, 5000 Forbes Ave, Pittsburgh, PA 15213}
\email{teddy@stat.cmu.edu}
\urladdr{\url{https://www.cmu.edu/dietrich/philosophy/people/faculty/seidenfeld.html}}

\keywords{Constriction, dilation, sets of probabilities, evidence, conditioning, forgetting}
\subjclass[2010]{Primary: 62A01; Secondary: 60A10, 60A99}

\title{Constriction for sets of probabilities}

\usepackage[T1]{fontenc}
\usepackage{amsmath}
\usepackage{ccaption}
\usepackage{amssymb}
\usepackage{amsthm}
\usepackage[left=2.5cm, right=2.5cm, top=3cm]{geometry}
\usepackage{hyperref}
\usepackage{algorithm}
\usepackage{cjhebrew}
\usepackage{algpseudocode}
\usepackage{bbm}
\usepackage{tikz}
\usepackage{caption}
\usepackage{subcaption}
\usepackage{fancyhdr}
\usepackage[inline]{enumitem}
\usepackage{comment}
\usepackage{nicefrac}
\usepackage{bm}
\usepackage{mathrsfs}
\usepackage[russian,english]{babel}
\usepackage{graphicx}
\usepackage[utf8]{inputenc}
\usepackage{cancel}
\usepackage{mathtools}

\newcommand{\vertiii}[1]{{\left\vert\kern-0.25ex\left\vert\kern-0.25ex\left\vert #1 
    \right\vert\kern-0.25ex\right\vert\kern-0.25ex\right\vert}}

\AtBeginDocument{%
   \def\MR#1{}
}

\newtheorem{thm}{Theorem}[section]

\theoremstyle{definition} 
\newtheorem{defi}[thm]{Definition}
\let\olddefi\defi
\renewcommand{\defi}{\olddefi\normalfont}
\newtheorem{rmk}[thm]{Remark}
\let\oldrmk\rmk
\renewcommand{\rmk}{\oldrmk\normalfont}

\newtheorem{theorem}{Theorem}
\newtheorem{lemma}[theorem]{Lemma}
\newtheorem{proposition}[theorem]{Proposition}
\newtheorem{corollary}[theorem]{Corollary}

\newtheorem{definition}[theorem]{Definition}
\newtheorem{remark}[theorem]{Remark}
\newtheorem{example}[theorem]{Example}

\hypersetup{
    pdftitle={prove},
    pdfauthor={autori},
    pdfmenubar=false,
    pdffitwindow=true,
    pdfstartview=FitH,
    colorlinks=true,
    linkcolor=blue,
    citecolor=green,
    urlcolor=cyan
}

\providecommand{\MR}[1]{}

\providecommand{\MR}{\relax\ifhmode\unskip\space\fi MR }

\providecommand{\href}[2]{#2}

\begin{document}

\begin{abstract}
Given a set of probability measures $\mathcal{P}$ representing an agent's knowledge on the elements of a sigma-algebra $\mathcal{F}$, we can compute upper and lower bounds for the probability of any event $A\in\mathcal{F}$ of interest. A procedure generating a new assessment of beliefs is said to constrict $A$ if the bounds on the probability of $A$ after the procedure are contained in those before the procedure. It is well documented that (generalized) Bayes' updating does not allow for constriction, for all $A\in\mathcal{F}$. In this work, we show that constriction can take place with and without evidence being observed, and we characterize these possibilities.
\end{abstract}

\maketitle
\thispagestyle{empty}

\section{Introduction}\label{intro}
Call $\Delta(\Omega,\mathcal{F})$ the space of countably additive probability measures on a measurable space $(\Omega,\mathcal{F})$ of interest and let $\mathcal{P} \subseteq \Delta(\Omega,\mathcal{F})$ be a set of probability measures. Then $\underline{P}(A)=\inf_{P\in\mathcal{P}}P(A)$ is called the lower probability of $A$, and its conjugate $\overline{P}(A)=1-\inf_{P\in\mathcal{P}}P(A^c)=\sup_{P\in\mathcal{P}}P(A)$ is called the upper probability of $A$. They are two of the main building blocks of the literature known as \textit{imprecise probability theory} \cite{walley}. Like in measure theory, where if outer and inner measures of a set coincide, then we say that the set has a measure, if upper and lower probabilities coincide, then they are an ordinary probability measure, and $\mathcal{P}$ is a singleton. The parallel we just drew between imprecise probability theory and measure theory is not merely heuristic in nature: \cite[Section 3.1.5]{walley} shows that the natural extension of a coherent lower probability from an algebra to the power set is the corresponding inner measure.

The reasons for studying imprecise probabilities are discussed at length in \cite{augustin_ed, walley} and references therein; in this work we focus especially on the motivations expressed in \cite{ellsberg,marinacci2}. There, the authors point out how specifying sets of probabilities -- and thus their ``boundary elements'', namely lower and upper probabilities -- accounts for the ambiguity faced by the agent carrying out the analysis. This means that since the agent does not know the true data generating process governing the experiment of interest, they may want to take advantage of the flexibility of IP theory and specify a set of probability measures to represent their ignorance. The set will be ``wider'', that is, the difference between $\overline{P}(A)$ and $\underline{P}(A)$ will be larger for all $A\in\mathcal{F}$, the higher the uncertainty faced by the agent.



The aim of this paper is to study the constriction phenomenon that takes place after a given procedure. 
\begin{definition}\label{contr_def}
Consider an event of interest $A\in\mathcal{F}$, a generic set of probability measures $\mathcal{P}\subseteq \Delta(\Omega,\mathcal{F})$, and denote by $\underline{P}$ and $\overline{P}$ the lower and upper probabilities associated with $\mathcal{P}$, respectively. Call \cjRL{no} a generic procedure that produces a new assessment of beliefs, and denote by $\underline{P}^\text{\cjRL{no}}$ and $\overline{P}^\text{\cjRL{no}}$ the lower and upper probabilities resulting from such procedure, respectively.\footnote{We use Hebrew letter \cjRL{no} to denote the procedure because the Hebrew word for procedure, \cjRL{nohal} (pronounced \textit{nohal}), begins with \cjRL{no}. In addition, Latin, Greek, and Cyrillic letters $p,P,\pi,\Pi$ -- that could be associated with the word ``procedure'' -- are usually associated with probabilities and partitions, while Greek letter $\varpi$ can be easily confused with $\omega$, which we will use to denote an element of the state space $\Omega$ of interest.} Then, we say that procedure \cjRL{no} \textit{(strictly) uniformly constricts} $A$, in symbols $\text{\cjRL{no}} \looparrowleft A$, if $\underline{P}^\text{\cjRL{no}}(A)>\underline{P}(A)$ and $\overline{P}^\text{\cjRL{no}}(A)<\overline{P}(A)$. We say that \cjRL{no} \textit{weakly uniformly constricts} $A$ if one of the two inequalities is weak.
\end{definition}
Trivially, if \cjRL{no} strictly uniformly constricts $A$, then \cjRL{no} weakly uniformly constricts $A$. In the remainder of the paper, we refer to uniform constriction (UC) simply as ``constriction''. UC is sometimes called contraction \cite{prob.kin,gong,agm}. We prefer constriction -- as denoted in \cite{herron} -- because contraction is used in the belief revision literature to denote an instance of corrigibility for full beliefs, which happens when an agent gives up some current evidence by moving to a logically weaker body of evidence \cite{levi3}. In addition when we say that \cjRL{no} produces a new assessment of beliefs, we mean that procedure \cjRL{no} outputs (a set of) probabilities that represent the belief of the agent around the elements of $\mathcal{F}$. This should not be confused with AGM theory \cite{agm} where procedures generate a new set of full beliefs (sets of sentences). We keep the same terminology as, given the context, no confusion arises.

Our interest for constriction stems from surprising results involving the opposite phenomenon, called \textit{dilation}, which was first observed in the context of (generalized) Bayes' updating of $\mathcal{P}$ \cite{seidenfeld_dil}. We remark that $\mathcal{P}$ need not be closed or convex. Pick any $A\in\mathcal{F}$, and call $\mathcal{P}(A):=\{P(A): P\in\mathcal{P}\}$. Let $X:\Omega \rightarrow \mathbb{R}$ be a $\mathcal{P}$-measurable random variable, that is, let it be $P$-measurable for all $P\in\mathcal{P}$. Call then $\mathcal{I}$ a generic index set, and let $\mathbf{X}:=\{X=x_i\}_{i\in\mathcal{I}}$ be the sample space of measurable events associated with $X$. Denote by $\mathcal{P}(A\mid x_i):=\{P(A\mid X=x_i): P\in\mathcal{P}\}$ the set of conditional probabilities of event $A$, given $X=x_i$. 
In order to avoid issues with conditional probability given a $P$-null event, $P\in\mathcal{P}$, we assume that the $P$'s in $\mathcal{P}$ agree on those outcomes $\{X=x_i\}$ that are $P$-null. Let us denote by $E=X^{-1}(x)\subseteq\Omega$ the evidence collected after experiment $\mathbf{X}$. A weaker notion of constriction is the following. 
\begin{definition}\label{str_pw_constr}
    The (generalized) Bayes' updating procedure, denoted by $\text{\cjRL{no}}=(B,E)$,\footnote{In $\text{\cjRL{no}}=(B,E)$, letter $B$ denotes generalized Bayes' updating, and $E=X^{-1}(x)$ is the conditioning set. Throughout the paper, we refer to generalized Bayes' updating simply as ``conditioning'', while other techniques are referred to as ``updating rules''.} strictly pointwise constricts $A$ if, for each $x_i$ in a set of $\mathcal{P}$-probability $1$ (that is, a set of ${P}$-probability $1$, for all $P\in\mathcal{P}$), $\underline{P}(A) < \underline{P}(A\mid X=x_i)$ and $\overline{P}(A) > \overline{P}(A\mid X=x_i)$. It weakly pointwise constricts $A$ if one of the two inequalities is weak.

    Generalized Bayes' updating merely pointwise constricts $A$ if it pointwise constricts $A$, but $\inf_{i\in\mathcal{I}} \underline{P}(A\mid X=x_i)=\underline{P}(A)$ and $\sup_{i\in\mathcal{I}} \overline{P}(A\mid X=x_i)=\overline{P}(A)$. Strict and weak mere pointwise constriction are defined similarly to before.
\end{definition}
To see that pointwise constriction is weaker than uniform constriction, notice that we obtain strict UC if $\inf_{i\in\mathcal{I}} \underline{P}(A\mid X=x_i)>\underline{P}(A)$ and $\sup_{i\in\mathcal{I}} \overline{P}(A\mid X=x_i)<\overline{P}(A)$, and weak UC if one of the latter two inequalities is weak. 

Define now, for all $P\in\mathcal{P}$, $\mathbf{X}^{A+}_P:=\{x \in\mathbf{X}:P(A\mid X=x)>P(A)\}$ and $\mathbf{X}^{A-}_P:=\{x \in\mathbf{X}:P(A\mid X=x)<P(A)\}$. The following lemma comes immediately from the law of conditional expectations.


\begin{lemma}\label{lemma_cond_bayes}
Pick any $A\in\mathcal{F}$. Then, for all $P\in\mathcal{P}$, $P(\mathbf{X}^{A+}_P)>0$ if and only if  $P(\mathbf{X}^{A-}_P)>0$.
\end{lemma} 


We then have two propositions that give sufficient conditions that preclude even weak constriction when collecting evidence in the form of an experiment to learn the value of the random variable $X$ and using generalized Bayes' updating.

\begin{proposition}\label{prop_bayes}
Pick any $A\in\mathcal{F}$. If $\mathcal{P}(A)$ is closed in the Euclidean topology, then no experiment $\mathbf{X}$ is such that $(B,E)$ weakly uniformly constricts $A$.
\end{proposition}

\begin{proposition}\label{prop_bayes2}
Pick any $A\in\mathcal{F}$. No simple experiment $\mathbf{X}$ is such that $(B,E)$ weakly pointwise constricts $A$. That is, if $X$ is a simple random variable (i.e. if the index set $\mathcal{I}$ for the sample space $\mathbf{X}$ is finite with $P$-probability $1$, for all $P\in\mathcal{P}$), then $(B,E)$ does not weakly pointwise constrict $A$.
\end{proposition}

The next example shows that if $\mathcal{P}(A)$ is open and $X$ is not simple, we can have strict mere pointwise constriction. Together with Propositions \ref{prop_bayes} and \ref{prop_bayes2}, this exhausts the possible cases for generalized Bayes' updating.

\begin{example}\label{ex_pw_constr}
    Let $A$ be a (measurable) event. For $0.4 < x < 0.6$ stipulate that $P_x(A) = x$, and let $\mathcal{P}(A) = \{P_x(A)\}$. So, $\underline{P}(A) = 0.4$ and $\overline{P}(A) = 0.6$, but $\mathcal{P}(A)$ is an open set. Let $f$ be a well ordering of the rational numbers in the open interval $(0.4, 0.6)$, denoted as the set $\mathbb{Q}_{(0.4,0.6)}$. So letting $\mathbb{N}=\{1,2,\ldots\}$ denote the natural numbers, we have $f: \mathbb{Q}_{(0.4,0.6)} \leftrightarrow \mathbb{N}$ is a $1-1$ (onto) function. Denote by $q_n=f^{-1}(n)$. Again, $q_n \in \mathbb{Q}_{(0.4,0.6)}$. Let $N$ be a $P_X$-measurable random variable where for each $0.4 < x < 0.6$, the likelihood ratio satisfies
\begin{equation}\label{lik_rat}
    \frac{P_x(N=n \mid A)}{P_x(N=n \mid A^c)}=\frac{(1-x)q_n}{x(1-q_n)}.
\end{equation}
Note that \eqref{lik_rat} constraints the distribution $P_x(N)$ without defining it. But \eqref{lik_rat} is coherent since, for each $0.4 < x < 0.6$, there are infinitely many values of $q_n$ for which the likelihood ratio is greater than $1$, and infinitely many values $q_n$ for which the ratio is less than $1$.\footnote{Here we see where the condition that $\mathcal{P}(A)$ is an open set is necessary. Condition \eqref{lik_rat} is incoherent when $P_x(A) = 0.4$ or $P_x(A) = 0.6$. Then, for all values of $q_n$, the likelihood ratio \eqref{lik_rat} would have values only to one side of $1$, in contradiction with the law of total probability.} By a trivial application of Bayes’ Theorem,
$$\frac{P_x(A \mid N=n)}{P_x(A^c \mid N=n)}=\frac{q_n}{1-q_n},$$
which is constant over $\mathcal{P}(A)$. That is, with respect to set $\mathcal{P}(A)$, the family of conditional probabilities of $A$, given $N = n$, is determinate despite the fact that the family of unconditional probabilities of $A$ is indeterminate. Thus, for each $n\in\mathbb{N}$,
$$0.4 < \underline{P}(A \mid N=n) = q_n = \overline{P}(A \mid N=n) < 0.6$$ and Bayes' updating, given $N=n$, strictly merely pointwise constricts $A$. Note well that the strict constriction is not uniform over $\mathbb{N}$ as $\inf_{n\in\mathbb{N}} \underline{P}(A \mid N=n) = 0.4 = \underline{P}(A)$, and $\sup_{n\in\mathbb{N}} \overline{P}(A \mid N=n) = 0.6 = \overline{P}(A)$.
\end{example}

Let us now give an example of dilation, borrowed from \cite{seidenfeld_dil}.  It illustrates how, using generalized Bayes' updating, imprecise probabilities for an event $A$ increase imprecision, for each possible outcome of an experiment.

\begin{example}
Suppose we flip a fair coin twice so that the flips may be dependent. Denote by $H_i$ and $T_i$ outcome ``heads'' and ``tails'', respectively, in tosses $i\in\{1,2\}$. Let 
$$\mathcal{P}:=\left\lbrace{P : P(H_1)=P(H_2)=\frac{1}{2} \text{, } P(H_1 \cap H_2)=p}\right\rbrace_{p\in\left[0,\frac{1}{2}\right]}.$$
Now, suppose we flip the coin; we have $P(H_2)=1/2$, but 
$$0=\underline{P}^B(H_2\mid H_1)<\underline{P}(H_2)=\frac{1}{2}=\overline{P}(H_2)<\overline{P}^B(H_2\mid H_1)$$
and
$$0=\underline{P}^B(H_2\mid T_1)<\underline{P}(H_2)=\frac{1}{2}=\overline{P}(H_2)<\overline{P}^B(H_2\mid T_1),$$
where $\overline{P}^B(H_2\mid H_1)=\overline{P}^B(H_2\mid T_1)=1$. As we can see, we start with a precise belief about the second toss and, no matter what the outcome of the first toss is, we end up having vacuous beliefs about the second toss.
\end{example}

The fact that Bayes' rule of conditioning -- arguably the most popular beliefs updating procedure -- can give rise to dilation is one motivation for exploring updating techniques that instead admit constriction. 
That is the focus of our work.

The paper is divided as follows. Section \ref{no_data} studies procedures that allow constriction to take place when no new evidence is collected. Theorems \ref{generic-no-data} and \ref{generic-no-data-cor} are the main results and give very general conditions for procedures to give the opportunity for constriction in the absence of new collected evidence. In section \ref{hybrid}, each individual in a group applies a (convex) personal pooling rule with ``precise'' inputs from the others in order to form their revised opinion.  The process iterates until the individual opinions merge to a fixed point.  Because the pooling rules are convex, the fixed point is a constriction of the original set of opinions. 
Section \ref{data} studies constriction when evidence is collected and non-Bayesian updating rules are used to revise the agent's beliefs. For a countably additive probability, given a generic partition $\mathcal{E}$ of $\Omega$, conditioning does not allow constriction for all $E\in\mathcal{E}$. So the only way of obtaining constriction for all  $E\in\mathcal{E}$ is to intentionally forget the experiment associated with $\mathcal{E}$. But if we are able to make assumptions about the nature of $\underline{P}$, we can give conditions for constriction to take place for all $E\in\mathcal{E}$. Section \ref{concl} concludes our work. We study opportunities for constriction when we forego the assumption of countably additive probabilities in Appendix \ref{fin-add}, and we prove our results in Appendix \ref{proof}.


\section{Constricting without evidence}\label{no_data}
In this section, we study procedures that give the opportunity for constriction when no data are collected. 
\subsection{Coherent extension of a precise probability}\label{coh_ext}
Recall that, for de Finetti, a probability measure $P$ is coherent if for any finite collection $\{A_i\}_{i=1}^n$ of nonempty subsets of a state space $\Omega$ of interest, we have that $\sup_{\omega\in\Omega}\sum_{i=1}^n c_i [I_{A_i}(\omega)-P(A_i)] \geq 0$, for all $c_1,\ldots,c_n\in\mathbb{R}$, where $I_{A_i}$ denotes the indicator function for set $A_i$. De Finetti's Fundamental Theorem of Probability \cite[Section 3.10]{definetti1} is the following.
\begin{theorem}\label{def_coh1}
Call $\Omega$ the state space of interest. Given the probabilities $P(A_i)$ of a finite number of events $A_1,\ldots,A_n \subseteq \Omega$, the probability $P(A_{n+1})$ of a further event $A_{n+1}$ 
\begin{enumerate}
    \item either turns out to be determined if $A_{n+1}$ is linearly dependent on the $A_i$'s;
    \item or can be assigned, coherently, any value in a closed interval $[p^\prime,p^{\prime\prime}]$.
\end{enumerate}
More precisely, $p^\prime$ is the greatest lower bound (GLB) $\sup P(X)$ of the evaluations from below of the $P(X)$ given by the random quantities $X$ linearly dependent on the $A_i$'s for which we certainly have $X \leq A_{n+1}$.\footnote{This inequality has to be interpreted as $X(\omega) \leq I_{A_{n+1}}(\omega)$, for all $\omega\in\Omega$.} 
The same can be said for $p^{\prime\prime}$ (replacing $\sup$ by $\inf$, maximum by minimum, $A^\prime_{n+1}$ by $A^{\prime\prime}_{n+1}$, and changing the direction of the inequalities, etc. It is the least upper bound of evaluations from above).
\end{theorem}
Notice that $[p^\prime,p^{\prime\prime}]$ can be an illusory restriction, for example if $p^\prime=0$ and $p^{\prime\prime}=1$. The interpretation to this result is the following. Suppose we express our subjective beliefs around events $A_1,\ldots,A_n$ via a precise probability distribution $P$. The fact that $P$ is precise is a crucial tenet of de Finetti's subjective probability theory. Then, if we want to coherently extend our beliefs to a new event $A_{n+1}$ of interest, either we can do that ``for free'' if $A_{n+1}$ is a linear combination of the other events, or we have an interval $[p^\prime,p^{\prime\prime}]$ within which to select the value to assign to $P(A_{n+1})$. De Finetti himself does not say specifically how to choose a value within $[p^\prime,p^{\prime\prime}]$. The takeaway seems to be along the lines of ``you should be able to think hard enough to come up with a precise number $p\in[p^\prime,p^{\prime\prime}]$ to attach to $P(A_{n+1})$''. 

Denote by $\text{\cjRL{no}}=\text{deFin}$ the procedure of choosing any value in $[p^\prime,p^{\prime\prime}]$ to assign to the probability of event $A_{n+1}$. Then, the following holds.
\begin{theorem}\label{def_contr}
Suppose -- in the notation of Theorem \ref{def_coh1} -- that $p^\prime \neq p^{\prime\prime}$. Then, $\text{deFin} \looparrowleft A_{n+1}$ if $P^\text{deFin}(A_{n+1})\in(p^\prime,p^{\prime\prime})$; the constriction is weak if $P^\text{deFin}(A_{n+1})\in\{p^\prime,p^{\prime\prime}\}$.
\end{theorem}

Recall that, for Walley, a lower probability measure $\underline{P}$ is coherent if for any finite collection $\{A_i\}_{i=0}^n$ of nonempty subsets of a state space $\Omega$ of interest, we have that
$$\sup_{\omega\in\Omega}\left[ \sum_{i=1}^n   \left(I_{A_i}(\omega)-\underline{P}(A_i) \right)-s \left(I_{A_0}(\omega)-\underline{P}(A_0) \right) \right] \geq 0,$$
for all $s,n\in\mathbb{Z}_+$. In \cite[Section 3.1]{walley} the author gives the imprecise probabilities (IP) counterpart of Theorem \ref{def_coh1}. That is, Walley presents a method to extend coherently lower and upper probabilities $\underline{P}(A_i)$, $\overline{P}(A_i)$ from a finite collection of sets $\{A_1,\ldots,A_n\}\subseteq 2^\Omega$ to any other $A_{n+1} \subseteq \Omega$. However, this result is not intended to prompt constriction, in contrast with de Finetti's Fundamental Theorem. The central idea in IP theory is to be ``comfortable'' working with sets of probabilities, and not being forced to select a precise value inside the set.


Notice that Theorems \ref{def_coh1} and \ref{def_contr}, and the results in \cite[Section 3.1]{walley}, are given in de Finetti’s and Walley’s frameworks, respectively. They both rely on the finitely additive probabilities. Because we consider finitely many events, though, this distinction is immaterial. That being said, it is important to point out that de Finetti’s and Walley’s extensions procedures apply also starting from an arbitrary (possibly infinite) set of events, where the distinction between finite and countable additivity matters, see for example Appendix \ref{fin-add}.

Let $\#$ denote the cardinality operator, $\text{Conv}(H)$ the convex hull of a generic set $H$, and $\text{ex}[K]$ the extreme points of a generic convex set $K$.
We can generalize Theorem \ref{def_contr} to the following.
\begin{theorem}\label{generic-no-data}
Suppose a generic procedure $\text{\cjRL{no}}$ generates a set $\mathcal{P} \subseteq \Delta(\Omega,\mathcal{F})$ of probabilities on $(\Omega,\mathcal{F})$ such that $\#\mathcal{P}\geq 2$, and then prescribes a way of selecting one element $P^\star=\underline{P}^\text{\cjRL{no}}=\overline{P}^\text{\cjRL{no}}$ from $\text{Conv}(\mathcal{P})$. Assume that $\emptyset \neq \text{ex}[\text{Conv}(\mathcal{P})]=\{P^{ex}_j\}_{j\in\mathcal{J}}$, where $\mathcal{J}$ is a generic index set. 
Then, we have that
\begin{itemize}
    \item if $P^\star\in\text{ex}[\text{Conv}(\mathcal{P})]$, then 
    there may exist a collection $\{\tilde{A}\}\subseteq \mathcal{F}$ for which $\text{\cjRL{no}}$ weakly constricts $\tilde{A}$. In addition, $\text{\cjRL{no}} \looparrowleft A$, for all $A\in\mathcal{F}\setminus\{\tilde{A}\}$;
    \item if instead $P^\star=\sum_{j\in\mathcal{J}} \alpha_j P^{ex}_j$, $\alpha_j>0$ for all $j$, 
    then $\text{\cjRL{no}} \looparrowleft A$, for all $A\in\mathcal{F}$.
\end{itemize}
\end{theorem}

We can also give a topological version of Theorem \ref{generic-no-data}; call $\partial_X H$ and $\text{int}_X H$ the boundary and the interior of a generic set $H$ in $X$, respectively. 

\begin{theorem}\label{generic-no-data-cor}
Endow $[0,1]$ with the Euclidean topology, and call $\mathcal{B}([0,1])$ the Borel sigma-algebra on $[0,1]$. Fix a generic $A\in\mathcal{F}$, and assume that $\mathcal{P}(A):=\{P(A): P\in\mathcal{P}\} \subseteq \mathcal{B}([0,1])$ and that $\#\mathcal{P}(A)\geq 2$.  Then, 
\begin{itemize}
  \item if $\mathcal{P}(A)$ is closed in the Euclidean topology and $P^\star(A)\in\partial_{\mathcal{B}([0,1])}\mathcal{P}(A)$, then $\text{\cjRL{no}}$ weakly constricts ${A}$;
\item if instead $P^\star(A)\in\text{int}_{\mathcal{B}([0,1])}\mathcal{P}(A)$, then $\text{\cjRL{no}} \looparrowleft A$.
\end{itemize}
\end{theorem}

\begin{remark}\label{rem_topology}
    Notice that 
    the assumption that $\mathcal{P}(A) \subseteq \mathcal{B}([0,1])$ is verified in the case that $\mathcal{P}$ is convex.
\end{remark}

It is immediate to see how Theorem \ref{def_contr} is a special case of Theorems \ref{generic-no-data} and \ref{generic-no-data-cor}. Another procedure that fits the requirement of Theorems \ref{generic-no-data} and \ref{generic-no-data-cor} is \textit{Halmos' extension} \cite[Exercise 48.4]{halmos}, \cite[Section 4.13]{billingsley}. Consider two generic measurable spaces $(X,\mathcal{X})$ and $(Y,\mathcal{Y})$. Let $\mu\in\Delta(X,\mathcal{X})$ and, for all $x\in X$, $\nu_x\in\Delta(Y,\mathcal{Y})$. Suppose further that for all $B\in\mathcal{Y}$, $\nu_{\bullet}(B):X \rightarrow [0,1]$ is $\mathcal{X}$-measurable. Then,
\begin{itemize}
    \item[(i)] map $E\mapsto \nu_x(\{y\in Y : (x,y)\in E\})$ is $\mathcal{X}$-measurable, for all $E\in\mathcal{X}\times\mathcal{Y}$;
    \item[(ii)] map $\pi:\mathcal{X}\times\mathcal{Y} \rightarrow [0,1]$,
    $$E \mapsto \pi(E):=\int_X \nu_x \left(\{y\in Y : (x,y)\in E\}) \mu(\text{d}x\right)$$
    is a probability measure on $\mathcal{X}\times\mathcal{Y}$.
\end{itemize}
Suppose now that there exists a set $A\subseteq Y$ such that its inner and outer measures do not coincide, that is, for all $x\in X$, ${\nu_\star}_x(A) \neq \nu^\star_x(A)$. Then, consider $A^\prime=X\times A$. We have that
$$\pi_{A^\prime}\equiv \pi(A^\prime)=\int_X \nu_x \left(\{y\in Y : (x,y)\in A^\prime\}) \mu(\text{d}x\right),$$
and $\pi_{A^\prime} \in [{\nu_\star}_x(A),\nu^\star_x(A)]$. So Halmos' extension prescribes a way to extend a countable additive probability measure on $\mathcal{Y}$ to another countably additive probability measure on $\mathcal{X}\times\mathcal{Y}$ that gives a well defined measure to a $\mathcal{Y}$-non-measurable set $A$. This value belongs to the interval whose endpoints are the inner and outer measures of $A$, respectively. It is immediate to see, then, how Halmos' extension satisfies the conditions of Theorems \ref{generic-no-data} and \ref{generic-no-data-cor}.

\begin{remark}\label{difference_defin_halmos}
Before going on, we need to mention a noteworthy difference between $\text{\cjRL{no}}=\text{deFin}$ and extension theorems from measure theory (à la Halmos). The relevant contrast is that for the Fundamental Theorem (applied to probability), de Finetti uses as his domains linear spans of, e.g. indicator functions. And for the measure theorists, the extension of probabilities to a larger ring of sets uses (countable) sums of indicators defined: in the finite case from an algebra, and in the infinite case from a sigma-algebra. Two additional meaningful differences are (1) for finite structures, de Finetti does not require that probabilities are defined over an algebra, whereas, the others do; (2) for infinite structures, as de Finetti does not require countable additivity, his  inner and outer approximations are by finite sums of indicators. By contrast, the measure theorists require countably additive probabilities, and so they use countable sums for constructing inner and outer measure approximations.
\end{remark}

We now briefly present three procedures that fit the requirements of Theorems \ref{generic-no-data} and \ref{generic-no-data-cor}:
\begin{itemize}
    \item[(i)] \textit{convex pooling} \cite[Section 2]{ojea}, in which the opinions of $k$ agents (expressed via precise probabilities $P_j$, $j\in\{1,\ldots,k\}$) are first pooled in a convex way, thus forming a set
    \begin{align*}
        \mathcal{P}=\left\lbrace{P\in\Delta(\Omega,\mathcal{F}) : P=\sum_{j=1}^k \zeta_j P_j }\right\rbrace,
    \end{align*}
    where $\zeta_j\geq 0$, for all $j \in\{1,\ldots,k\}$ and $\sum_{j=1}^k\zeta_j=1$, from which a unique pooled opinion $P^\star$ is selected;
    \item[(ii)] \textit{Jaynes' MaxEnt} \cite{jaynes}, in which, given a set of constraints $\mathbf{C}$, the set of probabilities of interest to the researcher is $\mathcal{P}=\{P\in\Delta(\Omega,\mathcal{F}) : P \text{ satisfies } \mathbf{C}\}$, and $P^\star$ is selected by maximizing the Shannon entropy in $\mathcal{P}$;
    \item[(iii)] \textit{generalized fiducial inference} (GFI) \cite{hannig}, in which a set of data-dependent measures on the parameter space $\Omega$ -- called generalized fiducial distributions -- is defined by carefully inverting a deterministic data-generating equation without the use of Bayes’ theorem. Mathematically, we can write $\mathcal{P}=\{P\in\Delta(\Omega,\mathcal{F}) : P \text{ satisfies \cite[Equation (2)]{hannig}} \}$. As pointed out in \cite[Remark 4]{hannig}, $P^\star$ is then selected by choosing the appropriate norm to endow the sample space. In \cite[Section 1]{hannig}, the authors point out how, while GFI is different philosophically from Dempster-Shafer theory \cite{dsc} and inferential models \cite{martin}, the resulting solutions of these three methods are often mathematically closely related to one another.


\end{itemize}

\section{Constricting based on convex pooling}\label{hybrid}
As an example of a procedure that allows agents to collect new evidence, but does not 
use conditioning to update an agent's beliefs, we present the famous model in \cite{degroot}. There, the author supposes that there are $k$ individuals, each having their own subjective probability distribution $F_i$ for the unknown value of some parameter $\omega\in\Omega$.\footnote{Usually the elements of the parameter space $\Theta$ are denoted by $\theta$, while the elements of the state space $\Omega$ by $\omega$. Since the focus of DeGroot's model is the parameter space only, we used -- just in this section -- the $\omega\in\Omega$ notation for the parameter space to maintain the notation consistent with other sections.} For agent $i$, the opinions of all the other $k-1$ agents represent new evidence. Instead of conditioning on those, agent $i$ pools their own opinion with that of the other agents. DeGroot shows that, repeating this process for all agent $i$, the group reaches (asymptotically) an agreement on a common subjective probability distribution. After updating their opinions, the  probability distribution for every member of the group belongs to the set
\begin{equation}\label{eq_cvx_pool}
    \mathcal{P}=\left\lbrace{F=\sum_{j=1}^k \zeta_{j}F_{j}}\right\rbrace,
\end{equation}
where $\zeta_{j}\in[0,1]$, for all $j\in\{1,\ldots,k\}$, and $\sum_{j=1}^k \zeta_{j}=1$. In particular, for all $i\in\{1,\ldots,k\}$, we write that after the first (pooling) iteration, the updated probability measure for agent $i$, denoted by $F_{i1}$ is given by
$F_{i1}=\sum_{j=1}^k p_{ij}F_{j}$. This means that individual $i$ weighs the opinion of all the agents, including themselves, via coefficients $p_{i1},\ldots,p_{ik}$ representing the relative importance that agent $i$ assigns to the opinion of the other members of the group. Because this is true for all agents, we can give a linear algebra notation to the updating process. Call $\mathbf{P}$ the $k\times k$ stochastic matrix whose rows are given by probability vectors $(p_{i1},\ldots,p_{ik})$, $i\in\{1,\ldots,k\}$. Call then $\mathbf{F}=(F_1,\ldots,F_k)^\top$; we have that $\mathbf{F}^{(1)}=\mathbf{P}\mathbf{F}$, where $\mathbf{F}^{(1)}:=(F_{11},\ldots,F_{k1})^\top$. Of course this holds for all iterations, so in turn we have that $\mathbf{F}^{(n)}=\mathbf{P}\mathbf{F}^{(n-1)}=\mathbf{P}^n\mathbf{F}$, for all $n\in\mathbb{N}$. The members continue to make these revisions indefinitely or until $\mathbf{F}^{(n)}=\mathbf{F}^{(n-1)}$, for all $n\geq N$, for some $N\in\mathbb{N}$, so further revisions would not change the opinions of the members. The following is the main result of \cite{degroot}.
\begin{theorem}\label{degr_cons}
If there exists $n\in\mathbb{N}$ such that every element in at least one column of $\mathbf{P}^n$ is positive, then a consensus is reached.
\end{theorem}
That is, if the condition in Theorem \ref{degr_cons} is satisfied, then there exists a $k\times 1$ dimensional vector $\boldsymbol{\pi}=(\pi_1,\ldots,\pi_k)$ (that is unique, as guaranteed by \cite[Theorem 3]{degroot}) whose elements are non-negative and sum up to $1$, and such that $\boldsymbol{\pi}\mathbf{P}=\boldsymbol{\pi}$. In turn, this entails that if we call $\boldsymbol{\Pi}$ the $k\times k$ stochastic matrix whose rows are all the same and equal to $\boldsymbol{\pi}$, we have that $\mathbf{F}^\star=\boldsymbol{\Pi}\mathbf{F}$, where $\mathbf{F}^\star=(F_1^\star,\ldots,F_k^\star)^\top$ such that $F^\star=F_1^\star=\cdots F_k^\star=\sum_{j=1}^k \pi_j F_j$, where $F^\star$ is the common subjective distribution that is reached in the consensus. Notice that $F^\star$ belongs to $\mathcal{P}$ in \eqref{eq_cvx_pool}.

Call $\mathcal{P}_n:=\text{Conv}(F_{1n},\ldots,F_{kn})$, for all $n\in\mathbb{N}$. 
In this example we have that $\mathcal{P}_n \subseteq \mathcal{P}_{n-1}$, for all $n\in\mathbb{Z}_+$, where $\mathcal{P}_0$ is set $\mathcal{P}$ in equation \eqref{eq_cvx_pool}. This means that the limit of sequence $(\mathcal{P}_n)$ is set $\cap_{n\in\mathbb{Z}_+}\mathcal{P}_n$. If the condition in Theorem \ref{degr_cons} is satisfied, then $\cap_{n\in\mathbb{Z}_+}\mathcal{P}_n=\{F^\star\}$. Given a generic set $A\in\mathcal{F}$, DeGroot procedure $\text{\cjRL{no}}=\text{DeGr}$ may only weakly constrict $A$ in general, for example if $F^\star$ belongs to the extrema of $\mathcal{P}\equiv\mathcal{P}_0$. That is, $F^\star(A)\geq \underline{F}(A)$ and $F^\star(A)< \overline{F}(A)$, or $F^\star(A)> \underline{F}(A)$ and $F^\star(A)\leq \overline{F}(A)$. Nevertheless, there may exist $r,s\in\mathbb{Z}_+$, $r<s$, such that $\underline{F}_s(A) > \underline{F}_r(A)$ and $\overline{F}_s(A)< \overline{F}_r(A)$, where $\underline{F}_s(A)=\inf_{F\in\mathcal{P}_s}F(A)$, $\overline{F}_s(A)=\sup_{F\in\mathcal{P}_s}F(A)$, and similarly for $\mathcal{P}_r$.


DeGroot model is one of the possible examples of an agent collecting evidence and then revising their initial opinion using a rule that is different from conditioning on the gathered data. We showed that there is at least one such procedure in which constriction can take place if all the gathered information is used. 


Notice that, in contrast with deGroot's model for consensus, in the light of Proposition \ref{prop_bayes}, there is no opportunity for constriction in Aumann's important ``Agreeing to Disagree'' model \cite{aumann}. We outline the reason why here.  Identify the group's initial IP set of probabilities $\mathcal{P}(A)$ about the event of interest $A$, with the lower and upper probabilities $\underline{P}(A)$ and $\overline{P}(A)$, taken with respect to minimum and maximum of the precise individual opinions about $A$ that results after each agent learns their ``private'' information.  That is, $\mathcal{P}(A)$ obtains at round 1 in Aumann's process. Since there are finitely many agents in the group, this IP set is closed. In Aumann's model, the agents then iteratively share their individual, precise probabilities about $A$.  At each subsequent round after the first, they use Bayes' updating to revise their individual probability of $A$, given the new, shared evidence of what they learn about the other's probability of $A$.  This procedure amounts to (iteratively) using generalized Bayes' updating for the closed IP set $\mathcal{P}(A)$ given the updated individual precise probabilities for event $A$.  By Aumann's theorem, after finitely many rounds the process reaches a fixed point where a consensus $P^\star(A)$ is reached.  But if $\underline{P}(A)<\overline{P}(A)$, then Proposition \ref{prop_bayes} establishes that it cannot be that the consensus opinion, $P^\star(A)$, always satisfies $\underline{P}(A)< P^\star(A) <\overline{P}(A)$.

\section{Constricting based on non-Bayesian updating}\label{data}
Suppose the results of an experiment induce a partition $\mathcal{E}=\{E_j\}$ of the state space of interest $\Omega$. Then, if we retain the assumption that probability measures are countably additive, conditioning on $E_j$ does not allow for constriction, for all $E_j\in\mathcal{E}$. So in general we have that constriction for all $E_j\in\mathcal{E}$ can take place only if we intentionally forget the whole experiment that induces partition $\mathcal{E}$. But if we are able to make assumptions on the nature of lower probability $\underline{P}$ associated with the set $\mathcal{P}$ of probabilities representing the agent's beliefs, and if we consider updating procedures that are alternative to Bayes' conditioning, we have opportunities for constriction. 

\subsection{Background}\label{backgr}
In this section, we give some background concepts that are needed to better understand the results that follow. Lower and upper probabilities (LP and UP, respectively) 
are a particular type of Choquet capacities.
\begin{definition}\label{choq_def}
Given a measurable space $(\Omega,\mathcal{F})$ with $\Omega \neq \emptyset$, we say that a set function $\nu:\mathcal{F} \rightarrow [0,1]$ is a \textit{Choquet capacity} if $\nu(\emptyset)=0$, $\nu(\Omega)=1$, and $\nu(A) \leq \nu(B)$ for all $A,B \in \mathcal{F}$ such that $A \subseteq B$.
\end{definition}

Denote by $\mathcal{M}:=\{P\in\Delta(\Omega,\mathcal{F}) : P(A)\geq \underline{P}(A) \text{, } \forall A \in \mathcal{F}\}$ the set of (countably additive) probability measures \textit{compatible} with $\underline{P}$ \cite{gong}, and assume it is nonempty and relatively compact. Notice also that $\mathcal{M}$ is convex. 
The following are special cases of lower probabilities. 
\begin{definition}\label{k-order}
LP $\underline{P}$ is a \textit{Choquet capacity of order $k$}, or \textit{$k$-monotone capacity}, if for every collection $\{A, A_1, \ldots , A_k\} \subseteq \mathcal{F}$ such that $A_i \subseteq A$, for all $i \in \{1, \ldots, k\}$, we have
\begin{equation}\label{k-mon}
    \underline{P}(A) \geq \sum_{\emptyset\neq \mathcal{I} \subseteq \{1, \ldots, k\}} (-1)^{\#\mathcal{I} -1} \underline{P}(\cap_{i\in \mathcal{I}} A_i).
\end{equation}
Its conjugate UP $\overline{P}$ is called a \textit{$k$-alternating capacity} because it satisfies that for every collection $\{A, A_1, \ldots , A_k\} \subseteq \mathcal{F}$ such that $A \subseteq A_i$, for all $i \in \{1, \ldots, k\}$,
\begin{equation}\label{k-alt}
    \overline{P}(A) \leq \sum_{\emptyset\neq \mathcal{I} \subseteq \{1, \ldots, k\}} (-1)^{\#\mathcal{I} -1} \overline{P}(\cup_{i\in \mathcal{I}} A_i),
\end{equation}
\end{definition}
A special case of LP that we will use in the remainder of the paper are \textit{convex LP's}, that are Choquet capacities of order $2$; they satisfy $\underline{P}(A\cup B) \geq \underline{P}(A)+\underline{P}(B)-\underline{P}(A\cap B)$, for all $A,B\in\mathcal{F}$. Another special case of  LP that we will use are belief functions.
\begin{definition}\label{belief}
A LP $\underline{P}$ is called a \textit{belief function} if it is a Choquet capacity of order $\infty$, i.e., if \eqref{k-mon} holds for every $k$.
\end{definition}

Unique to a belief function is  its intuitive interpretation as a random set object that realizes itself as subsets of $\Omega$.
\begin{definition}\label{mass_funct_def}
If $\underline{P}$ is a belief function, its associated \textit{mass function} is the non-negative set function $m:\mathcal{F} \rightarrow [0,1]$,
\begin{equation}\label{mass_funct}
     A \mapsto m(A):=\sum_{B\subseteq A} (-1)^{\#(A-B)} \underline{P}(B),
\end{equation}
where $A-B \equiv A \cap B^c$, and the subsets $B$ of $A$ have to belong to $\mathcal{F}$ as well.
\end{definition}
Properties of mass function $m$ are the following
\begin{itemize}
    \item[(a)] $m(\emptyset)=0$;
    \item[(b)] $\sum_{B\subseteq\Omega} m(B)=1$;
    \item[(c)] $\underline{P}(A)=\sum_{B\subseteq A} m(B)$, and is unique to $\underline{P}$.
\end{itemize}
Formula \eqref{mass_funct} is called the Möbius transform of $\underline{P}$ \cite{yager}. A mass function $m$ induces a precise probability distribution on $\mathcal{F}$, as the distribution of a random set. These concepts are further studied in \cite{gong}. Notice also that Definition \ref{mass_funct_def} only applies if $\Omega$ is a finite set. A general definition and Möbius characterization of belief functions on infinite sets has been given in \cite{davide_barbara}.

To update a set of probabilities $\mathcal{P}$ given a set $E \in\mathcal{F}$ is to replace set function $\underline{P}$ with a version of the conditional set function $\underline{P}^{\times}(\cdot \mid E)$. The definition of $\underline{P}^{\times}$
is precisely the job of the updating rule. Recall that we introduced generalized Bayes' rule of conditioning $\text{\cjRL{no}}=(B,E)$ in section \ref{intro}.

We now give the formal definitions of three additional updating rules for lower and upper probabilities. Generalized Bayes', geometric, and Dempster's rules are the ones that are most commonly used and studied in the literature, while Gärdenfors' rule is a very general updating mechanism that subsumes many other methods of belief revision. The reasons for why an agent endorses one instead of another are explored in \cite{gong,smets2}.
\begin{definition}\label{geom-def}
Let $\mathcal{P} \subseteq \Delta(\Omega,\mathcal{F})$ be closed and convex. Then, the conditional LP's and UP's according to the \textit{geometric rule} are set functions $\underline{P}^G$, $\overline{P}^G$ such that, for all $A,E\in\mathcal{F}$
\begin{align}\label{geom}
\begin{split}
\underline{P}^{(G,E)}(A)&\equiv\underline{P}^G(A\mid E):=\frac{\underline{P}(A\cap E)}{\underline{P}(E)} \quad \text{and}\\ 
    \overline{P}^{(G,E)}(A)&\equiv\overline{P}^G(A\mid E)=1-\underline{P}^{(G,E)}(A^c),
    \end{split}
\end{align}
provided that $\underline{P}(E)>0$.
\end{definition}
So the main difference between generalized Bayes' and geometric updating procedures is that the former considers the infimum of the ratio of $P(A\cap E)$ and $P(E)$, while the latter considers the ratio of the infima. We introduce next Dempter's updating rule.
\begin{definition}\label{dempster-def}
Call $E \in\mathcal{F}$ the collected evidence. Assume that $\underline{P}$ is a belief function having mass function $m$ and such that $\underline{P}(E)>0$. Let $\underline{P}_0$ be a separate belief function whose associated mass function $m_0$ is such that $m_0 (E) = 1$. The conditional belief function $\underline{P}^D(\cdot\mid E)$ is defined as 
$$\underline{P}^{(D,E)}(A)\equiv\underline{P}^D(A\mid E):=\underline{P}(A)\oplus \underline{P}_0(E), \quad \forall A\in\mathcal{F},$$
where combination operator $\oplus$ means that the mass function associated with $\underline{P}^D(\cdot\mid E)$ is
$$m^D(A\mid E)=\frac{\sum_{C\cap E=A} m(C)}{\sum_{C^\prime\cap E\neq\emptyset} m(C^\prime)}, \quad \forall A\in\mathcal{F}.$$
Consequently, Dempster’s updating rule yields the following. If $\mathcal{P}=\mathcal{M}$, then the LP's and UP's according to \textit{Dempster’s updating rule} are set functions $\underline{P}^D$, $\overline{P}^D$ such that, for all $A,E\in\mathcal{F}$
\begin{align}\label{demp}
\begin{split}
    \overline{P}^{(D,E)}(A)&\equiv\overline{P}^D(A\mid E):=\frac{\overline{P}(A\cap E)}{\overline{P}(E)} \quad \text{and}\\
    \underline{P}^{(D,E)}(A)&\equiv\underline{P}^D(A\mid E)=1-\overline{P}^{(D,E)}(A^c),
\end{split}
\end{align}
provided that $\overline{P}(E)>0$.
\end{definition}
If $\underline{P}$ is a belief function and $\mathcal{P}=\mathcal{M}$, the geometric rule appears to be a natural dual to Dempster’s rule. Operationally, though, they differ, as pointed out in \cite[Section 2.2]{gong}. The main difference is that Dempster’s rule requires $\underline{P}$ to be a belief function, while geometric rule does not. A thorough comparison of geometric and Dempster's rules can be found in \cite{dubois_prade,gilboa_beliefs}. In addition, an axiomatic extension of Dempster’s conditioning rule to $2$-monotone/$2$-alternating capacities has been given in \cite{submodular}. 

Finally, we introduce the following. 
\begin{definition}\label{garden}
Call $E \in\mathcal{F}$ the collected evidence. Assume that $\underline{P}$ is a belief function having mass function $m$. Consider a function $\mathfrak{f}:\mathcal{F}\times \mathcal{F}\rightarrow [0,1]$ having constraints
\begin{itemize}
    \item[(a)] $\sum_{B\in\mathcal{F}}\mathfrak{f}(B,X)=1$, for all $X\in\mathcal{F}$,
    \item[(b)] $B\subseteq E^c \implies \mathfrak{f}(B,X)=0$,
    \item[(c)] $\mathfrak{f}(\emptyset,X)=0$, for all $X\in\mathcal{F}$.
\end{itemize}
Then, the belief function $\underline{P}^{(I,E)}(\cdot)\equiv\underline{P}^I(\cdot\mid E)$ obtained according to \textit{Gärdenfors' generalized imaging updating rule} (GGI) is such that its associated mass function is given by
\begin{equation}\label{garden-eq}
    m^I(A\mid E)=\sum_{X\in\mathcal{F}} \mathfrak{f}(A,X) m(X), \quad \forall A\in\mathcal{F}.
\end{equation}
\end{definition}
Since $E$ is the collected information, we have that $E\sqcup E^c=\Omega$, where $\sqcup$ denotes the disjoint union. Equation \eqref{garden-eq} tells us that, upon learning that the evidence collected is not in $E^c$, then the probabilities given by $m$ to the (sub)events in $E^c$ are transferred to the ``closest'' (sub)events in $E$ according to function $\mathfrak{f}$. Constraint (a) is needed to ensure that the ``probability bits'' that are transferred are then normalized. Constraint (b) corresponds to the closed world assumption, which can be expressed as the statement ``if the evidence collected is not in $E^c$, then it must belong to $E$''. Smets \cite[Section C.6]{smets2} drops requirement (b) because he works under the open world assumption, which negates the previous statement to symbolize that the agent may have specified the state space they work with incorrectly. So if the evidence collected is not in $E^c$, it may be in $E$ but also in a superset of $E$ that the agent did not consider at the beginning of the experiment. Constraint (c) is just a sanity check. Function $\mathfrak{f}$ is a version of a conditional probability; it was first introduced in \cite{garden} and then generalized by \cite{smets2}. The choice of function $\mathfrak{f}$ informs how having collected evidence $E$ 
influences our change of beliefs around $A$. GGI subsumes many other existing updating rules \cite{smets2}.

In the remainder of the paper, we write $\text{\cjRL{no}}=(\times,E)$, $\times\in\{B,G,D,I\}$, to indicate the generalized Bayes', geometric, Dempster's, and Gärdenfors' updating rules, respectively, given collected evidence $E$. We write that $(\times,E)\looparrowleft A$ if given evidence $E$, rule $\times$ (strictly) constricts $A$, and $(\times,\mathcal{E})\looparrowleft A$ if rule $\times$ (strictly) constricts $A$ for all elements of partition $\mathcal{E}$.

\subsection{Intentional forgetting}\label{int_for}
In this section we show the following claim. If we are not willing to assume that $\underline{P}$ is at least convex -- let alone a belief function -- that is, if we can only use generalized Bayes' and geometric rules to update our beliefs, then we cannot obtain constriction for all the elements $E$ of a partition $\mathcal{E}$ representing the results of an experiment of interest. In this very general case, we need to forget in order to constrict. We retain the assumptions that $\mathcal{P}$ is closed and convex.
\begin{lemma}\label{lem-no-contr}
Let $\mathcal{E}$ be a measurable and denumerable partition of $\Omega$, and let $\times\in\{B,G\}$. Then, for any $A\in\mathcal{F}$ we have that 
$$\inf_{E\in\mathcal{E}} \underline{P}^\times(A\mid E) \leq \underline{P}(A) \quad \text{and} \quad \sup_{E\in\mathcal{E}} \overline{P}^\times(A\mid E) \geq \overline{P}(A).$$
\end{lemma}

An immediate consequence of Lemma \ref{lem-no-contr} is the following.
\begin{theorem}\label{thm-no-contr}
Let $\mathcal{E}$ be a measurable and denumerable partition of $\Omega$. Then for $\times\in\{B,G\}$, we have that for all $A\in\mathcal{F}$, there exists $E\in\mathcal{E}$ such that $\text{\cjRL{no}}=(\times,{E})$ does not weakly constrict $A$.
\end{theorem}
This result tells us that for any event $A\in\mathcal{F}$ of interest, we can never find an element of $\mathcal{E}$ that (even weakly) constricts $A$. In turn, this implies that there exist $E_1,E_2\in\mathcal{E}$, $E_1$ possibly different than $E_2$, such that $[\underline{P}(A),\overline{P}(A)] \subset [\underline{P}^B(A\mid E_1),\overline{P}^B(A \mid E_1)]$ and $[\underline{P}(A),\overline{P}(A)] \subset [\underline{P}^G(A\mid E_2),\overline{P}^G(A \mid E_2)]$. 
As it appears clear, if the agent is unwilling to make any extra assumption on the nature of $\mathcal{P}$, then there is no opportunity for constriction to take place for all the elements of partition $\mathcal{E}$.  

In \cite[Theorem 2.3]{seidenfeld_dil} the authors give sufficient conditions for dilation to take place for all $E\in\mathcal{E}$. Then, intentionally forgetting altogether the experiment that dilates $A$ seems the only viable option to reach constriction. As the name suggests, intentional forgetting corresponds to an agent willingly forgetting pieces of information, for example because they are redundant, because they may be harmful, or because they are instructed to do so. If after collecting evidence $E$ our current beliefs are encapsulated in lower and upper probabilities $\underline{P}^\times(\cdot\mid E)$ and $\overline{P}^\times(\cdot\mid E)$, respectively, then by forgetting we mean reversing the learning process so that our ``updated'' lower and upper probabilities becomes what used to be the lower and upper ``priors'', i.e. $\underline{P}(\cdot)$ and $\overline{P}(\cdot)$, respectively.

The topic of forgetting is studied in statistics. In \cite{moens}, for example, forgetting is intended in the sense of \textit{stabilized forgetting}; with this we mean the following. Suppose that the agent is operating in an environment that is susceptible to changes. 
Then, the agent's response to surprising events depends on their beliefs about how likely the environment is to change. If it is volatile, a single unexpected event triggers forgetting of past beliefs and relearning of a new contingency. So at time $t$, the agent collects evidence $E_t$; they use it to infer whether the environment has changed or not. In the former case, they erase their memory of past events and reset their prior belief to their initial prior knowledge. In the latter, they can learn a new posterior belief of the environment structure based on their previous belief. Another example is given by limited memory procedures. In \cite{baudry,barry}, the authors study bandit problems based on limited memory: working with restricted memory, data that are too old are forgotten.

Forgetting is studied in machine learning (ML) as well. In \cite{cao}, for instance, the authors come up with an algorithm that features a forgetting factor which balances the relative importance of new data and past data and adjust the model to pay more attention to the new data when the concept drift is detected. In this framework, forgetting is intended as in past data progressively losing importance as new evidence is collected. Another example in ML where forgetting is crucial is \textit{continual learning}: as data gets discarded and has a limited lifetime, the ability to forget what is not important and retain what matters for the future are the main issues that continual learning targets and focuses on \cite{lesort}.

Finally, and rather unsurprisingly, psychologists and cognitive scientists have thoroughly inspected the phenomenon of forgetting. The reference textbook that investigates intentional forgetting is \cite{int_for}. The authors examine the effect on memory of instructions to forget in a wide variety of contexts.  They point out how with the enormous number of information available nowadays, online forgetting of some information is necessary, and how often times replacing existing information with new information is mandatory (think of a person changing their phone number). Study on intentional forgetting stemmed from the phenomenon of \textit{directed forgetting}: we are able to deal more effectively with large amounts of information by following instructions to treat some of the information as ``to be forgotten'' (e.g. evidence presented in a courtroom that, being inadmissible, is asked to be disregarded). In this way, interference is reduced and we are able to devote all of our resources to the remaining to-be-remembered information. It is easy to see how stabilized forgetting is a particular case of intentional forgetting, and so is the forgetting factor approach used in the machine learning literature.

Consider the problem of an agent that expresses their initial beliefs on $(\Omega,\mathcal{F})$ via a set of probabilities $\mathcal{P}=\mathcal{M}$. For convenience, we write $\mathcal{P} \equiv \mathcal{P}_{E_0}$. As data become available, they update their beliefs using Bayes' rule of conditioning for every element of $\mathcal{P}$. With this we mean the following. Suppose we collect evidence in the form of $E_1 \subseteq \mathcal{F}$; then we update the elements of $\mathcal{P}$ to obtain
\begin{align*}
    \mathcal{P}_{E_1}:=\bigg\{&P_{E_1}\in\Delta(\Omega,\mathcal{F}) : P_{E_1}(A)\equiv P(A\mid E_1)\\
    = &\frac{P(E_1\mid A) P(A)}{P(E_1)}\propto P(E_1\mid A) P(A) \text{, } \forall A\in\mathcal{F} \bigg\},
\end{align*}
where $P\in \mathcal{P}$ represents the prior and $P(E_1\mid \cdot)$ represents the likelihood. More in general, let the evidence collected up to time $t>0$ be encapsulated in collection $\{E_k\}_{k=1}^t \subseteq \mathcal{F}$. Then, the agent's updated opinion is given by set 
\begin{align*}
    \mathcal{P}_{E_1\cdots E_t}&:=\bigg\{P_{E_1\cdots E_t}: P_{E_1\cdots E_t}(A)\equiv P_{E_1\cdots E_{t-1}}(A\mid E_t)\\ 
    &\propto P(E_t\mid A,E_1,\ldots,E_{t-1}) P_{E_1\cdots E_{t-1}}(A) \text{, } \forall A\in\mathcal{F} \bigg\},
\end{align*}
where $\mathcal{P}_{E_1\cdots E_{t-1}}\ni P_{E_1\cdots E_{t-1}} (\cdot) \equiv P(\cdot \mid E_1\cdots E_{t-1})$ represents the ``revised'' prior (that is, the posterior computed at time $t-1$) and $P(E_t\mid \cdot)$ represents the likelihood.

Pick any $k\in\mathbb{N}$ such that $1\leq k\leq t$. Assume that $\mathcal{P}_{E_1\cdots E_{t-k}}$ and $\mathcal{P}_{E_1\cdots E_{t-k} \cdots E_t}$ are both convex and closed.
Fix an event $A\in\mathcal{F}$ of interest and define ${\mathcal{P}_\star}_{E_1\cdots E_{t-k}}(A):=\{P_{E_1\cdots E_{t-k}}\in\mathcal{P}_{E_1\cdots E_{t-k}} : P_{E_1\cdots E_{t-k}}(A)=\underline{P}_{E_1\cdots E_{t-k}}(A)\}$ and ${\mathcal{P}^\star}_{E_1\cdots E_{t-k}}(A):=\{P_{E_1\cdots E_{t-k}}\in\mathcal{P}_{E_1\cdots E_{t-k}} : P_{E_1\cdots E_{t-k}}(A)=\overline{P}_{E_1\cdots E_{t-k}}(A)\}$. For a generic $P\in\Delta(\Omega,\mathcal{F})$, and generic $A,B\in\mathcal{F}$, define notion of dependence $d_P$ and sets induced by its value by
\begin{align*}
    d_P(A,B)&:=P(A\cap B)-P(A)P(B),\\
    \Sigma^+(A,B)&:=\{P\in\Delta(\Omega,\mathcal{F}):d_P(A,B)>0\},\\
    \Sigma^-(A,B)&:=\{P\in\Delta(\Omega,\mathcal{F}):d_P(A,B)<0\}.\\
\end{align*}
Call now $\text{\cjRL{no}}=(\text{IF}_\times,\mathscr{E})$ the procedure of intentionally forgetting evidence $\mathscr{E}$ (where $\mathscr{E}$ can be an element of $\mathcal{F}$, a whole partition $\mathcal{E}$ of state space $\Omega$, or a collection $\{\mathcal{E}\}$ of partitions) after having updated endorsing rule $\times$. Then, the following gives sufficient conditions for intentional forgetting to induce constriction.

\begin{theorem}\label{suff_cond_contr}
Fix an event $A\in\mathcal{F}$ of interest and let the agent endorse any rule $\times\in\{B,G\}$. If
$${\mathcal{P}_\star}_{E_1\cdots E_{t-k}}(A) \cap \Sigma^-(A,E_{t-k+1}\cap \cdots \cap E_t) \neq \emptyset$$
and
$${\mathcal{P}^\star}_{E_1\cdots E_{t-k}}(A) \cap \Sigma^+(A,E_{t-k+1}\cap \cdots \cap E_t) \neq \emptyset,$$
then forgetting $E_{t-k+1}\cap \cdots \cap E_t$ strictly constricts $A$, in symbols $(\text{IF}_\times,E_{t-k+1}\cap\cdots\cap E_t)\looparrowleft A$.
\end{theorem}
If this holds for all elements $E_s$ of partition $\mathcal{E}_s$, $s\in\{t-k+1,\ldots,t\}$, we write 
$$(\text{IF}_\times,\mathcal{E}_{t-k+1},\ldots,\mathcal{E}_{t}) \looparrowleft A,$$
so we can forget all the experiments that took place after time $t-k$. If we let $k=t$, then we obtain stabilized forgetting as in \cite{moens}. Notice that for stabilized forgetting subscript $E_1\cdots E_{t-k}$ in Theorem \ref{suff_cond_contr} is substituted by $E_0$. We also have the following.
\begin{corollary}\label{suff_cond_contr:cor}
Fix an event $A\in\mathcal{F}$ of interest and let the agent endorse any rule $\times\in\{B,G\}$. If
$${\mathcal{P}_\star}_{E_k\cdots E_{t}}(A) \cap \Sigma^-(A,E_{1}\cap \cdots \cap E_{k-1}) \neq \emptyset$$
and
$${\mathcal{P}^\star}_{E_k\cdots E_{t}}(A) \cap \Sigma^+(A,E_{1}\cap \cdots \cap E_{k-1}) \neq \emptyset,$$
then forgetting $E_{1}\cap \cdots \cap E_{k-1}$ strictly constricts $A$, in symbols $(\text{IF}_\times,E_{1}\cap \cdots \cap E_{k-1}) \looparrowleft A$.
\end{corollary}
In this case, we obtain the machine learning version of forgetting \cite{cao}. The agent intentionally forgets data collected before time $k$. More formally, the authors weight the evidence $E_t$ collected at each time $t$ by a coefficient depending on $t$ that goes to $0$ the farther time $t$ is from present time $T$, that is, it goes to $0$ as $|t-T|$ grows to infinity. Evidence that is old enough gets severely discounted, to the point that for practical purpose we can consider it as being forgotten, and so the result in Corollary \ref{suff_cond_contr:cor} applies. 

Notice that intentional forgetting is always a viable way of inducing constriction, as long as the selected updating rule induces dilation first. In this section we focused on $\times\in\{B,G\}$ because generalized Bayes' and geometric rules are the most general ones we presented (they do not require the lower probability of interest to be a belief function).

\subsubsection{Levi-Neutrality}\label{levi}
A particular type of forgetting is the one inspired by Levi's work on corrigible infallibility, see e.g. \cite{levi2}. Suppose that at time $t$ an agent is equipped with a body of beliefs $\mathcal{K}_t$ regarding the events in $\mathcal{F}$, that is, a collection of logical predicates that describe the beliefs of the agent at time $t$.\footnote{Not to be confused with belief functions.} For the purpose of studying constriction, we focus on the (relatively compact and convex) set $\mathcal{P}_{\mathcal{K}_t}$ of probability measures representing the agent's beliefs at time $t$ induced by $\mathcal{K}_t$.\footnote{In particular, we assume $\mathcal{P}_{\mathcal{K}_t}=\mathcal{M}_{\mathcal{K}_t}:=\{P\in\Delta(\Omega,\mathcal{F}) : P(A)\geq \underline{P}_{\mathcal{K}_t}(A) \text{, } \forall A \in \mathcal{F}\}$.} Consider a generic event $H \subseteq \Omega$; every element ${P}_{\mathcal{K}_t}$ of $\mathcal{P}_{\mathcal{K}_t}$ has the following two properties:
\begin{itemize}
    \item if $\mathcal{K}_t$ rules out $H$, that is, if given the beliefs at time $t$ event $H$ is considered impossible -- written $\mathcal{K}_t \mapsto \neg H$ --, then $P_{\mathcal{K}_t}(A\mid H)$ is not well defined, for all $A\in\mathcal{F}$;
    \item if instead $\mathcal{K}_t$ does not rule out $H$, then $P_{\mathcal{K}_t}(A\mid H)$ is well defined, for all $A\in\mathcal{F}$, and represents the agent's beliefs around the plausibility of $A$ if $H$ obtains.
\end{itemize}
The set $\mathcal{P}_{\mathcal{K}_t}$ is compatible with lower probability $\underline{P}_{\mathcal{K}_t}$. Suppose then that at time $t+1$ the agent collects evidence $E$. Let $\mathcal{K}_{t+1}$ be the new body of beliefs; abusing notation, we write $\mathcal{K}_{t}\cup\{E\}$. Then, suppose that the agent endorses either of generalized Bayes' or geometric rules, $\times\in\{B,G\}$; the updated beliefs of the agent are encapsulated in $\underline{P}^\times_{\mathcal{K}_t}(\cdot\mid E)$. 
If (according to $\times$) $E$ dilates an event $A^\prime$ of interest, then the agent can neglect $E$ to obtain constriction; they run ``reverse conditioning''. We say that the agent is Levi-neutral towards $E$. Their infallible beliefs, encapsulated in $\mathcal{K}_{t+1}$, are subject to being corrected, and we have $\mathcal{K}_{t+2}=\mathcal{K}_{t}$. The body of beliefs contracts: $\mathcal{K}_{t+1}$ loses element $E$ and goes back to what used to be at time $t$. We write $(\text{LN}_\times,E)\looparrowleft A^\prime$ to denote that being Levi-neutral towards $E$ -- after having updated endorsing rule $\times$ -- constricts $A^\prime$. If this holds for all elements of a partition $\mathcal{E}$ of $\Omega$ representing the possible outcomes of an experiment of interest, we write $(\text{LN}_\times,\mathcal{E})\looparrowleft A^\prime$.

\subsection{Assumptions on the nature of $\underline{P}$}\label{ass-nature}
If the agent is willing to make some assumptions on the type of lower probability $\underline{P}$ that represents their beliefs, then we can have constriction for all $E\in\mathcal{E}$ without resorting to intentional forgetting. As the proverb goes, there is no free lunch. The following is Theorem 5.9 in \cite{gong}.
\begin{theorem}\label{geom-demps}
Let $\mathcal{E}=\{E,E^c\}$ be the partition associated with the outcomes of the experiment of interest. Assume that $\underline{P}$ is a belief function such that $\underline{P}(E),\underline{P}(E^c)>0$, and consider any event $A\in\mathcal{F}$. Then, if $\mathcal{E}$ dilates $A$ under the Geometric rule, then it must constrict $A$ under Dempster’s rule. 
Similarly, if $\mathcal{E}$ dilates $A$ under Dempster’s rule, then it must constrict $A$ under the Geometric rule.
\end{theorem}
The proof of Theorem \ref{geom-demps} only requires that $\underline{P}$ is convex, but we need the assumption that $\underline{P}$ is in fact a belief function otherwise we would not be able to use Dempster's rule (see Definition \ref{dempster-def}). As we can see, Dempster's and geometric rule contradict each other. 

Assuming that $\underline{P}$ is a belief function allows us to use updating rules that are otherwise inaccessible.
\begin{theorem}\label{imaging-theorem}
Let $E$ be the evidence collected by the agent, and assume that $\underline{P}$ is a belief function having mass function $m$ such that $\underline{P}(E)>0$. Consider any event $A\in\mathcal{F}$. We have that $(I,E) \looparrowleft A$ if and only if
\begin{align*}
    \sum_{B\subseteq A} &\left[\sum_{X\in\mathcal{F}} \mathfrak{f}(B,X)m(X) - m(B) \right]>0 \quad \text{and}\\
    \sum_{B\subseteq A^c} &\left[\sum_{X\in\mathcal{F}} \mathfrak{f}(B,X)m(X) - m(B) \right]>0.
\end{align*}
\end{theorem}

If the conditions in Theorem \ref{imaging-theorem} hold for every elements of partition $\mathcal{E}$, we write $(I,\mathcal{E}) \looparrowleft A$. 

The main point of this section is that if we are willing to formulate an assumption on the nature of lower probability $\underline{P}$ associated with set $\mathcal{P}$  representing our beliefs, then we are able to find constriction by using updating rules that in general do not allow for constriction for all $E\in\mathcal{E}$, like the geometric rule. We can also use entirely new updating techniques like Dempster's rule or GGI that are otherwise inapplicable.


\section{Conclusion}\label{concl}
In this paper, we show that, when updating an agent's opinions, there are at least three settings for constricting sets of probabilities (representing the beliefs), namely when belief revision is performed without evidence, when it is based on convex pooling, and when it is based on non-Bayesian updating. Also, we provide examples of procedures for every such framework. 

This is just the first step towards a deeper study of the constricting phenomenon, that we will carry over in the next future. In particular, we plan to find more instances in which constricting is possible, and to find a trait d'union linking these settings.

\section*{Acknowledgements}%
We would like to thank three anonymous referees for their helpgul and constructive comments. Michele Caprio would like to acknowledge funding from ARO MURI W911NF2010080.


\appendix
\section{Dubins-deFinetti conditioning}\label{fin-add}


In this section, we show how if we are willing to depart from the classical Kolmogorovian paradigm of probabilities, then we have additional opportunities for constriction. In particular, as we shall see, Bayes' rule can induce constriction if we allow probabilities to be merely finitely additive.
 
Suppose that we adopt Dubins-deFinetti conditioning (DdFC) framework; an in-depth exposition of DdFC can be found in \cite{defin3,dubins1,ragazzini_def_dub}. Notice also that in \cite{coletti16}, the authors compare generalized Bayes', geometric and Dempster's rules for belief functions in DdFC framework. For the sake of the present work, the two main differences with respect to the Kolmogorovian framework is that probabilities need not be countably additive, and that conditioning does not happen on sigma-fields, but rather on events, members of a partition of the state space. The following statements are true; we will provide illustrations for the first one, and the second one is shown similarly.



\begin{enumerate}
    \item If probability measures are finitely but not countably additive, then constriction can take place for all the elements of a countable partition;
    \item Recall that a probability $P$ is completely additive if the measurable union of a set of $P$-null events is $P$-null. If probability measures are countably but not completely additive, then constriction can take place for all the elements of an uncountable partition.
\end{enumerate}
Notice also that if probability measures are completely additive, then they must be discrete.

\begin{definition}\label{congl}
We say that probability measure $P$ is \textit{conglomerable in partition $\mathcal{E}$} when for every event $A$ such that $P(A\mid E)$ is defined for all $E\in\mathcal{E}$, and for all constants $k_1,k_2$, if  $k_1 \leq P(A\mid E)\leq k_2$ for all $E \in\mathcal{E}$, then $k_1 \leq P(A)\leq k_2$. 
\end{definition}
Definition \ref{congl} asserts that for each event $A$, if all the conditional probabilities over a partition $\mathcal{E}$ are bounded by two quantities, $k_1$ and $k_2$, then the unconditional probability for that event is likewise bounded by these two quantities \cite{fin_add_book}. De Finetti \cite{defin3} shows the non-conglomerability of finitely additive probability measures (FAPMs) in denumerable partitions.

Assume that instead of requiring $\mathcal{P}$ to be a set of countably additive probabilities, we allow it to be a set of  FAPMs. Then, weak and strict constriction can happen by Bayes updating $\mathcal{P}$ thanks to the non-conglomerability property of FAPMs. The two illustrations that we present in this section build on the example in \cite[page 92]{dubins1}, which we state here for motivating their construction.
\begin{example}\label{dubin_ex}
    \textbf{(Dubins)} Let $\Omega=\{A,B\}\times \{N=1,2,\ldots\}$. Stipulate that 
    \begin{itemize}
        \item $P(A)=P(B)=1/2$,
        \item $P(N=n\mid A)=2^{-n}$, for $n\in\{1,2,\ldots\}$, a countably additive conditional probability,
        \item $P(N=n\mid B)=0$, for $n\in\{1,2,\ldots\}$, a strongly finitely additive conditional probability.\footnote{Recall that a FAPM is strongly finitely additive if it admits countable partitions by null sets \cite{armstrong2}.}
    \end{itemize}
Then, $P(N=n)=2^{-(n+1)}>0$, for $n\in\{1,2,\ldots\}$, and (marginally) $P$ is merely finitely additive over the subalgebra generated by the partition $\mathcal{E}_N=\{\{N=1\},\{N=2\},\ldots\}$. $P$ displays non-conglomerability for the event $A$ in the partition $\mathcal{E}_N=\{\{N=1\},\{N=2\},\ldots\}$ as $P(A)=1/2$ and $P(A\mid N=n)=1$, for $n\in\{1,2,\ldots\}$.
\end{example}

\subsection*{Illustration 1 (weak constriction)}
Use Example \ref{dubin_ex} as follows. Consider a set $\mathcal{P}$ of probabilities on $\Omega=\{A,B\}\times \{N=1,2,\ldots\}$ such that $\mathcal{P}=\{P_\alpha, 0 < \alpha \leq 1\}$, where 
\begin{itemize}
    \item $P_\alpha(A)=\alpha$,
    \item $P_\alpha(N=n\mid A)=2^{-n}$, for $n\in\{1,2,\ldots\}$, a countably additive conditional probability,
    \item $P_\alpha(N=n\mid B)=0$, for $n\in\{1,2,\ldots\}$, a strongly finitely additive conditional probability. 
\end{itemize}
Note that neither $P_\alpha(N=n\mid A)$ nor $P_\alpha(N=n\mid B)$ depend upon $\alpha$. With respect to $\mathcal{P}$, we have $0<P_\alpha(A) \leq 1$. For each $0<\alpha <1$, we have non-conglomerability of $P_\alpha$ for the event $A$ in the partition $\mathcal{E}_N$ as $P_\alpha(A\mid N=n) = 1$, for $n\in\{1,2,\ldots\}$. Observe then that $1=P_1(A)=P_1(A\mid N=n)$, for $n\in\{1,2,\ldots\}$. Thus, Bayes-updating using the information $\{N =n\}$ from the partition $\mathcal{E}_N$ weakly-constricts $\mathcal{P}=\{P_\alpha, 0 < \alpha \leq 1\}$.

\subsection*{Illustration 2 (strict constriction)} 
Modify Example \ref{dubin_ex} as follows. Consider a set $\mathcal{P}$ of probabilities on $\Omega=\{A,B\}\times \{N=1,2,\ldots\}$ such that $\mathcal{P}=\{P_\alpha, 0 < \alpha < 1\}$, where 
\begin{itemize}
    \item $P_\alpha(A)=\alpha$, so that with respect to $\mathcal{P}$, we have $0<P_\alpha(A) < 1$,
    \item $P_\alpha(N=n\mid A)=(1-\alpha)2^{-n}$, for $n\in\{1,2,\ldots\}$,
    \item $P_\alpha(N=n\mid B)=\alpha 2^{-n}$, for $n\in\{1,2,\ldots\}$. 
\end{itemize}
Note that each of these two conditional probabilities is a merely finitely additive probability distribution over $N$ that depends on $\alpha$. In addition, observe that $P(N=n)=2\alpha(1-\alpha)2^{-n}>0$, which (for each $\alpha$) also is a merely finitely additive probability distribution over $N$.

By a Bayes' updating, for each $0 < \alpha < 1$ and each $n\in\{1,2,\ldots\}$, $P_\alpha(A\mid N=n)=1/2$. That is, for each $\alpha\neq 1/2$ with $0 < \alpha < 1$, there is non-conglomerability of $P_\alpha$ for the event $A$ in the partition $\mathcal{E}_N$. Whereas, $1/2 = P_{1/2}(A) = P_{1/2}(A\mid N=n)$. Thus, Bayes-updating using the information $\{N =n\}$ from the partition $\mathcal{E}_N$ strictly-constricts $\mathcal{P}=\{P_\alpha, 0 < \alpha < 1\}$.

\begin{remark}
    These two illustrations help to explain why Propositions \ref{prop_bayes} and \ref{prop_bayes2} are restricted to countably additive probabilities.
\end{remark}

\section{Proofs}\label{proof}

\begin{proof}[Proof of Proposition \ref{prop_bayes}]
Let $\underline{P}$ be a probability function that satisfies $\underline{P}(A) = \min_{P\in\mathcal{P}} P(A)$. By Lemma \ref{lemma_cond_bayes}, if $\underline{P}(\mathbf{X}^{A+}_{\underline{P}})>0$, then $\underline{P}(\mathbf{X}^{A-}_{\underline{P}})>0$ and $(B,E)$ does not strictly uniformly constrict $A$. That is, for each $P_{1i}(A)\in\mathcal{P}(A)$, $\underline{P}(A)\leq P_{1i}(A)$, so $(B,E)$ does not strictly uniformly constrict $A$ when $\underline{P}(A) = P_{1i}(A)$. Hence, if $(B,E)$ weakly uniformly constricts $A$, we have that $\underline{P}(\mathbf{X}^{A+}_{\underline{P}})=0$, and then $\underline{P}(\{x \in\mathbf{X}:\underline{P}^B(A\mid X=x)=\underline{P}(A)\})=1$. Let now $\overline{P}$ be a probability function that satisfies $\overline{P}(A) = \max_{P\in\mathcal{P}} P(A)$. By the same reasoning, $\overline{P}(\{x \in\mathbf{X}:\overline{P}^B(A\mid X=x)=\overline{P}(A)\})=1$, and then $(B,E)$ does not weakly uniformly constrict $A$ either.
\end{proof}

\begin{proof}[Proof of Proposition \ref{prop_bayes2}]
Assume (for a reductio proof) that on a set of $X$ values with $\mathcal{P}$-measure $1$ (i.e. with $P$-probability $1$, for all $P\in\mathcal{P}$), for each $x_i$ there exist $P_{1i}(A)$ and $P_{2i}(A)$ in $\mathcal{P}(A)$ such that for each $P(A\mid x_i)\in\mathcal{P}(A\mid x_i)$, either $P_{1i}(A) < {P}(A\mid x_i) \leq P_{2i}(A)$ or $P_{1i}(A) \leq {P}(A\mid x_i) < P_{2i}(A)$. Since $X$ is a simple random variable, define $P_1(A):=\min_i P_{1i}(A)$ and $P_2(A):=\max_i P_{2i}(A)$. Given Proposition \ref{prop_bayes}, assume $\mathcal{P}(A)$ is not a closed set. Without loss of generality, assume it is open 
below (the reasoning is parallel if $\mathcal{P}(A)$ is open above). So, $\underline{P}(A) < P_1(A)$. Then, there exists $P_0\in\mathcal{P}$ with $\underline{P}(A) < P_0(A) < P_1(A)$. Since for each $i$, $P_{1i}(A) < {P}_0(A\mid x_i)$ or $P_{1i}(A) \leq {P}_0(A\mid x_i)$, we also have that for each $i$, $P_0(A) < P_1(A) \leq {P}_0(A\mid x_i)$. But then $P_0(\mathbf{X}^{A+}_{P_0})=1$, which is a contradiction according to Lemma \ref{lemma_cond_bayes}.
\end{proof}

\begin{proof}[Proof of Theorem \ref{def_contr}]
Immediate from Definition \ref{contr_def} and Theorem \ref{def_coh1}.
\end{proof}

\begin{proof}[Proof of Theorem \ref{generic-no-data}]
We first show that the lower probability (LP) $\inf_{P\in \text{Conv}(\mathcal{P})}P(\cdot)$ of the convex hull of $\mathcal{P}$ and the LP $\inf_{P\in \text{ex}[\text{Conv}(\mathcal{P})]}P(\cdot)$ of the extrema of the convex hull of $\mathcal{P}$ coincide. To see this, pick any $A\in\mathcal{F}$. Since $\text{ex}[\text{Conv}(\mathcal{P})] \subseteq \text{Conv}(\mathcal{P})$, we have that
\begin{equation}\label{first_ineq_gnd}
    \inf_{P\in \text{Conv}(\mathcal{P})} P(A) \leq \inf_{P\in \text{ex}[\text{Conv}(\mathcal{P})]} P(A).
\end{equation}
Then, let $\emptyset\neq\text{ex}[\text{Conv}(\mathcal{P})]=\{P^{ex}_j\}_{j\in\mathcal{J}}$. For all $P\in \text{Conv}(\mathcal{P})$ and all $A\in\mathcal{F}$, we have that
\begin{align*}
    P(A)=\sum_{j\in\mathcal{J}} \alpha_j P^{ex}_j(A) &\geq \sum_{j\in\mathcal{J}} \alpha_j \underline{P}^{ex}(A)\\&=\underline{P}^{ex}(A):=\inf_{P\in \text{ex}[\text{Conv}(\mathcal{P})]} P(A),
\end{align*}
where $\{\alpha_j\}_{j\in\mathcal{J}}$ is a collection of positive reals such that $\sum_{j\in\mathcal{J}} \alpha_j=1$, which implies that
\begin{equation}\label{second_ineq_gnd}
    \inf_{P\in \text{Conv}(\mathcal{P})} P(A) \geq \inf_{P\in \text{ex}[\text{Conv}(\mathcal{P})]} P(A).
\end{equation}
By combining together \eqref{first_ineq_gnd} and \eqref{second_ineq_gnd} we obtained the desired equality.

Now, if 
$P^\star\in\text{ex}[\text{Conv}(\mathcal{P})]$, then there might be a collection $\{\tilde{A}\}\subseteq \mathcal{F}$ for which $P^\star(\tilde{A})=\underline{P}(\tilde{A})$ or $P^\star(\tilde{A})=\overline{P}(\tilde{A})$, so the constriction is weak for the elements of the collection, while $P^\star(A)>\underline{P}(\tilde{A})$ and $P^\star(A)<\overline{P}(\tilde{A})$, for all $A\in\mathcal{F}\setminus\{\tilde{A}\}$. If instead $P^\star=\sum_{j\in\mathcal{J}} \alpha_j P^{ex}_j$, $\alpha_j>0$ for all $j$, 
then $P^\star(A)\in(\underline{P}(A),\overline{P}(A))$, for all $A\in\mathcal{F}$, so we have $\text{\cjRL{no}} \looparrowleft A$, for all $A\in\mathcal{F}$.
\end{proof}

\begin{proof}[Proof of Theorem \ref{generic-no-data-cor}]
If $\mathcal{P}(A)$ is closed in the Euclidean topology and $P^\star(A)\in\partial_{\mathcal{B}([0,1])}\mathcal{P}(A)$, then $P^\star({A})=\underline{P}({A})$ or $P^\star({A})=\overline{P}({A})$, so the constriction is weak. If instead $P^\star(A)\in\text{int}_{\mathcal{B}([0,1])}\mathcal{P}(A)$, 
then $P^\star(A)\in(\underline{P}(A),\overline{P}(A))$, so we have $\text{\cjRL{no}} \looparrowleft A$.
\end{proof}

\begin{proof}[Proof of Lemma \ref{lem-no-contr}]
This proof draws on that of \cite[Lemma 5.1]{gong}. Fix any $A\in\mathcal{F}$. Because $\mathcal{P}$ is closed, there exists $P_{(A)} \in \mathcal{P}$ such that $P_{(A)}(A)=\underline{P}(A)$. Notice that subscript $(A)$ reminds us that this probability measure can vary with the choice of $A$. Then, we have that
\begin{align*}
    \underline{P}(A)=P_{(A)}(A)&=\sum_{E\in\mathcal{E}}P_{(A)}(A\mid E)P_{(A)}(E) \\
    &\geq \sum_{E\in\mathcal{E}}\underline{P}^\times(A\mid E)P_{(A)}(E)\\
    &\geq \sum_{E\in\mathcal{E}}\inf_{E\in\mathcal{E}}\underline{P}^\times(A\mid E)P_{(A)}(E)\\
    &=\inf_{E\in\mathcal{E}}\underline{P}^\times(A\mid E)\sum_{E\in\mathcal{E}}P_{(A)}(E)\\
    &=\inf_{E\in\mathcal{E}}\underline{P}^\times(A\mid E).
\end{align*}
The same argument applies for the upper probability of $A$, that is, if we pick $P^\prime_{(A)} \in \mathcal{P}$ such that $P^\prime_{(A)}(A)=\overline{P}(A)$, $P^\prime_{(A)}$ possibly different from $P_{(A)}$, then 
\begin{align*}
    \overline{P}(A) \leq \sum_{E\in\mathcal{E}}\overline{P}^\times(A\mid E)P^\prime_{(A)}(E) \leq \sup_{E\in\mathcal{E}} \overline{P}^\times (A\mid E).
\end{align*}
\end{proof}

\begin{proof}[Proof of Theorem \ref{thm-no-contr}]
Immediate from Lemma \ref{lem-no-contr}.
\end{proof}

\begin{proof}[Proof of Theorem \ref{suff_cond_contr}]
This proof comes from that of \cite[Theorem 2.3]{seidenfeld_dil}. Fix an event $A\in\mathcal{F}$ of interest, and let $\times\in\{B,G\}$. Pick any $P_{E_1\cdots E_{t-k}} \in {\mathcal{P}_\star}_{E_1\cdots E_{t-k}}(A) \cap \Sigma^-(A,E_{t-k+1}\cap\cdots\cap E_t)$. Then, we have that $P_{E_1\cdots E_{t-k}}(A)=\underline{P}_{E_1\cdots E_{t-k}}(A)$ because $P_{E_1\cdots E_{t-k}} \in {\mathcal{P}_\star}_{E_1\cdots E_{t-k}}(A)$ and $P_{E_1\cdots E_{t-k}}(A\cap E_{t-k+1}\cap\cdots\cap E_t)<{P}_{E_1\cdots E_{t-k}}(A){P}_{E_1\cdots E_{t-k}}(E_{t-k+1}\cap\cdots\cap E_t)$ because $P_{E_1\cdots E_{t-k}} \in \Sigma^-(A,E_{t-k+1}\cap\cdots\cap E_t)$. Then, 
\begin{align*}
    \underline{P}_{E_1\cdots E_{t-k}}(A) &= P_{E_1\cdots E_{t-k}}(A)\\
    &> \frac{P_{E_1\cdots E_{t-k}}(A\cap E_{t-k+1}\cap\cdots\cap E_t)}{{P}_{E_1\cdots E_{t-k}}(E_{t-k+1}\cap\cdots\cap E_t)}\\ 
    &= P_{E_1\cdots E_{t-k}}(A\mid E_{t-k+1}\cap\cdots\cap E_t)\\
    &\geq \underline{P}^\times_{E_1\cdots E_{t-k}}(A\mid E_{t-k+1}\cap\cdots\cap E_t).
\end{align*}
A similar argument gives us that $\overline{P}_{E_1\cdots E_{t-k}}(A)<\overline{P}^\times_{E_1\cdots E_{t-k}}(A\mid E_{t-k+1}\cap\cdots\cap E_t)$. So $E_{t-k+1}\cap\cdots\cap E_t$ dilates $A$ regardless of which updating rule $\times\in\{B,G\}$ the agent endorses. In turn, forgetting $E_{t-k+1}\cap\cdots\cap E_t$ constricts $A$, in symbols $(\text{IF}_\times,E_{t-k+1}\cap\cdots\cap E_t)\looparrowleft A$.
\end{proof}

\begin{proof}[Proof of Corollary \ref{suff_cond_contr:cor}]
Analogous to the proof of Theorem \ref{suff_cond_contr}.
\end{proof}

\begin{proof}[Proof of Theorem \ref{imaging-theorem}]
Fix any $A\in\mathcal{F}$. Recall that by \eqref{garden-eq}, we have that $m^I(A\mid E)=\sum_{X\in\mathcal{F}} \mathfrak{f}(A,X) m(X)$, for all $A\in\mathcal{F}$. Also, by property (c) of mass function $m$ associated to LP $\underline{P}$ (see Definition \ref{mass_funct}), we have that $\underline{P}(A)=\sum_{B\subseteq A}m(B)$. So,
$$\underline{P}^I(A\mid E)=\sum_{B\subseteq A}m^I(B\mid E)=\sum_{B\subseteq A}\sum_{X\in\mathcal{F}} \mathfrak{f}(B,X) m(X).$$

\textbf{Only if} Suppose $(I,E) \looparrowleft A$. Then, by definition of constriction, $\underline{P}^I(A\mid E)>\underline{P}(A)$ and $\overline{P}^I(A\mid E)<\overline{P}(A)$. This happens if and only if
\begin{align*}
    &\sum_{B\subseteq A}\sum_{X\in\mathcal{F}} \mathfrak{f}(B,X) m(X) > \sum_{B\subseteq A}m(B)\\
    &\iff \sum_{B\subseteq A} \left[\sum_{X\in\mathcal{F}} \mathfrak{f}(B,X)m(X) - m(B) \right]>0
\end{align*}
and
\begin{align*}
    &\sum_{B\subseteq A^c}\sum_{X\in\mathcal{F}} \mathfrak{f}(B,X) m(X) > \sum_{B\subseteq A^c}m(B)\\
    &\iff \sum_{B\subseteq A^c} \left[\sum_{X\in\mathcal{F}} \mathfrak{f}(B,X)m(X) - m(B) \right]>0.
\end{align*}
This latter is true because $\overline{P}^I(A\mid E)<\overline{P}(A) \iff 1-\underline{P}^I(A^c\mid E)<1-\underline{P}(A^c) \iff \underline{P}^I(A^c\mid E)>\underline{P}(A^c)$.
\textbf{If} Assume that 
\begin{align*}
    &\sum_{B\subseteq A} \left[\sum_{X\in\mathcal{F}} \mathfrak{f}(B,X)m(X) - m(B) \right]>0 \quad \text{and}\\
    &\sum_{B\subseteq A^c} \left[\sum_{X\in\mathcal{F}} \mathfrak{f}(B,X)m(X) - m(B) \right]>0.
\end{align*}
Then, we have
$$\underline{P}^I(A\mid E)=\sum_{B\subseteq A}\sum_{X\in\mathcal{F}} \mathfrak{f}(B,X) m(X) > \sum_{B\subseteq A}m(B) = \underline{P}(A)$$
and
\begin{align*}
    \underline{P}^I(A^c\mid E)=\sum_{B\subseteq A^c}\sum_{X\in\mathcal{F}} \mathfrak{f}(B,X) m(X) &> \sum_{B\subseteq A^c}m(B)\\
    &=\underline{P}(A^c),
\end{align*}
which implies $\overline{P}^I(A\mid E)<\overline{P}(A)$. This concludes the proof.
\end{proof}
 
\section*{Author Contributions}
Both of the authors contributed equally to this paper.

\bibliographystyle{plain}
\bibliography{ergodic_theory}

\end{document}